\documentclass[11pt]{article}

\usepackage{amsmath}
\usepackage{amsfonts}
\usepackage{amssymb}
\usepackage{amsthm}
\usepackage{breqn}
\usepackage{setspace}
\usepackage{fullpage}
\usepackage{enumitem}
\usepackage{comment}
\usepackage{hyperref}
 \usepackage{bbm}
\usepackage{tikz}
\usepackage{enumitem}
\usepackage{mathtools} 
\usepackage[utf8]{inputenc}
\usepackage[english]{babel}
\usepackage{mathtools}
\allowdisplaybreaks
\usetikzlibrary{patterns,arrows,decorations.pathreplacing}
\newtheorem{lemma}{Lemma}
\newtheorem{theorem}{Theorem}

\newtheorem{claim}{Claim}
\newtheorem{conjecture}{Conjecture}

\newcommand{\Z}{\mathbb{Z}}
\newcommand{\R}{\mathbb{R}}
\newcommand{\abs}[1]{\left\lvert{#1}\right\rvert}

\DeclareMathOperator{\ex}{ex}

\usepackage{authblk}
\begin{document}
\title{The Maximum Number of Paths of Length Four in a Planar Graph}
\author[1]{Debarun Ghosh}
\author[1,2]{Ervin Gy\H{o}ri} 
\author[3]{Ryan R. Martin} 
\author[1]{Addisu Paulos}
\author[1,2]{Nika Salia}
\author[1]{Chuanqi Xiao}
\author[1,4]{Oscar Zamora} 
\affil[1]{Central European University, Budapest\par
\texttt{oscarz93@yahoo.es, addisu_2004@yahoo.com\\chuanqixm@gmail.com, ghosh_debarun@phd.ceu.edu}}
\affil[2]{Alfr\'ed R\'enyi Institute of Mathematics, Budapest \par
\texttt{gyori.ervin@renyi.mta.hu, nika@renyi.hu }}
\affil[3]{Iowa State University, Ames, IA, USA, \par
\texttt{rymartin@iastate.edu}}
\affil[4]{Universidad de Costa Rica, San Jos\'e}
\maketitle
\begin{abstract}
Let $f(n,H)$ denote the maximum number of copies of $H$  in an $n$-vertex planar graph. 
The order of magnitude of $f(n,P_k)$, where $P_k$ is a path on $k$ vertices, is $n^{{\lfloor{\frac{k-1}{2}}\rfloor}+1}$.  
In this paper we determine the asymptotic value of $f(n,P_5)$ and give conjectures for longer paths.
\end{abstract}
\section{Introduction}
For a given graph $F$, the Tur\'an number $\ex(n,F)$ is the maximum number of edges in an $n$-vertex graph and containing no copy of $F$.  The generalized Tur\' an number $\ex(n, H, F)$ is the maximum number of copies of $H$ in an $n$-vertex  $F$-free graph.

A few examples of $\ex(n, H, F)$, with $H \neq K_2$, were studied first by and A.A.~Zykov in \cite{zykov1949some} and independently by P.~Erd\H os \cite{erdos1962number}. They determined $\ex(n, K_r, K_s)$ for all $r$ and $s$. Later E.~Gy\H ori, J.~Pach and M.~Simonovits \cite{gyoripach} studied $\ex(n, H, K_s)$ for various graphs $H$ when $s \geq 3$.
A different example that has received considerable attention recently is $\ex(n, C_r, C_s)$ for various
values of $r$ and $s$. In 2008, B.~Bollob\'as and E.~Gy\H ori \cite{bollobas2008pentagons} showed that $\ex(n, C_3, C_5) = \Theta (n^{3/2})$, and this paper was the start of a more extensive study of this type of problems.  Their
upper bound has been improved several times, meanwhile
E.~Gy{\H o}ri and H.~Li \cite{gyorili} obtained upper and lower
bounds on $\ex(n, K_3, C_{2k+1})$, that were subsequently improved by 
Z.~F\"uredi and L.~Ozkahya \cite{furediozkahya} by 
N.~Alon and C.~Shikhelman \cite{alon2016many,alon2018additive} . Moreover, the number $\ex(n, C_5, C_3)$ was determined precisely by H.~Hatami, J.~Hladk\'y, D.~Kr\' al', S.~Norine, and A.~Razborov \cite{hatami2013number} and independently by A.~Grzesik \cite{grzesik2012maximum}.
Very recently, L.~Gishboliner and A.~Shapira \cite{gishboliner2020generalized} determined $\ex(n, C_r, C_s)$, up to a constant factor, for
all $r,s>3$, and, additionally, they studied $\ex(n, C_3, C_s)$ for even $s$. Some additional more precise
estimates for $\ex(n, C_r, C_s)$ are known (see \cite{gerbner2020generalized,grzesik2018maximum}).
For more results, we refer the reader to \cite {ergemlidze2018triangles,ergemlidze2019note,gyori2018maximum}.



In this paper, we study generalized Tur\'an number of graphs in planar graphs. Let $f(n,H)$ be the maximum number of copies of $H$ in an $n$-vertex planar graph.   Equivalently such problems can be described as $\ex(n,H,\mathcal{F})$, where $\mathcal{F}$ is the family of subdivisions of $K_5$ and $K_{3,3}$ \cite{kuratowski1930probleme}.   S.~Hakimi and E.F.~Schmeichel \cite{hakimi1979number} determined the exact value of $f(n,H)$ when $H$ is a cycle of length $3$ and cycle of length $4$. 
Moreover, they gave a conjecture for the exact value of $f(n, H)$ when $H$ is a cycle of length five. Later E.~Gy\H{o}ri, A.~Paulos, N.~Salia, C.~Tompkins, and O.~Zamora confirmed it in \cite{gyHori2019maximum1}.  In the same paper, the order of magnitude is also given for $f(n, H)$ when $H$ is a cycle of length more than $4$.  

It is natural to ask the value of $f(n, P_{k})$, where $P_{k}$ is a path of  $k$ vertices.  
Clearly $f(n, P_2)=3n-6$ if $n\geq 3$.  
N.~Alon and Y.~Caro \cite{alon1984number} determined the exact value of $f(n, H)$, where $H$ is a complete bipartite graph in which the smaller class is of size $1$ or $2$.
Consequently from the former result, the value of $f(n, P_3)$ is determined.  
In particular they showed that   
\begin{theorem}\cite{alon1984number}
For $n\geq 4$, $f(n,P_3)=n^2+3n-16$. 
\end{theorem}
Recently, E.~Gy\H{o}ri, A.~Paulos, N.~Salia, C.~Tompkins, and O.~Zamora in \cite{gyHori2019maximum2} determined the exact value of $f(n, P_4)$. 
They proved the following results.  
\begin{theorem}\cite{gyHori2019maximum2}\label{thm1}
We have,
\begin{align*}
f(n,P_4)= \begin{cases}
12, &\text{if $n=4$;}\\
147, &\text{if $n=7$;}\\
222, &\text{if $n=8$;}\\
7n^2-32n+27, &\text{if $n=5, 6$ and $n\geq 9$.}
\end{cases}
\end{align*}
\label{main}
\end{theorem}
The same authors in \cite{gyHori2020generalized} also gave the order of magnitude of $f(n, P_{k})$.
\begin{theorem}\cite{gyHori2020generalized}
$f(n, P_{k})= \Theta(n^{{\lfloor{\frac{k-1}{2}}\rfloor}+1})$. 
\end{theorem}
In this paper we give an asymptotic value of $f(n, P_5)$. 
\begin{theorem}\label{4}
$f(n, P_5)=n^3+O(n^2)$.
\end{theorem}
This bound is asymptotically the best possible.  Consider the maximal planar graph on $n$ vertices containing two degree $n-1$ vertices as shown in Figure \ref{fig1}.  It can be checked that this graph contains at least $n^3$ copies of $P_5$. 
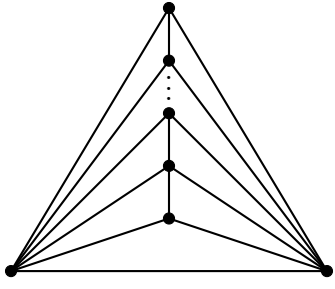
\begin{figure}[h]
\centering
\begin{tikzpicture}[scale=0.7]
\foreach \x in{1,2,3,4,5}{\draw[fill=black](0,\x)circle(3pt);};
\draw[fill=black](-3,0)circle(3pt);
\draw[fill=black](3,0)circle(3pt);
\foreach \x in{1,2,4}{\draw[thick](0,\x)--(0,\x+1);}
\draw (0,3.625)node{$\vdots$};
\foreach \x in{1,2,3,4,5}{\draw[thick](-3,0)--(0,\x)--(3,0);}
\draw[black,thick](-3,0)--(3,0);
\end{tikzpicture}
\caption{A graph on $n$ vertices showing  the  bound  in Theorem \ref{4}  is  asymptotically best possible.}
\label{fig1}
\end{figure}

Before we proceed to the proof of our result, we mention some notation.  For a graph  $G$, we use the notation $V(G)$ and $E(G)$ respectively for the vertex and edge sets of the graph.  The number of paths of length $4$ in $G$ is denoted by $P_5(G)$.  For a vertex $v\in V(G)$, the degree of $v$ is denoted by $d_G(v)$.  We may omit the subscript and write simply $d(v)$ if the underlying graph is clear.   Let $u$, $w\in V(G)$.  We denote the number of vertices in $G$ which are adjacent to both vertices by $d(u,w)$. 


\section{Proof of Theorem \ref{4}} 
For any given graph $G$ and vertices $u$ and $v$ in $G$, it is easy to see that the number of paths of length $4$ in the graph with $u$ and $v$ the two vertices next to the terminal vertices of the path is at most $d(u)d(u, v)d(v)$. Thus,  
\begin{align*}
P_5(G)\leq \frac{1}{2}\sum_{u\in V(G)}\sum_{u\neq v\in V(G)}d(u)d(v)d(u, v).
\end{align*}

Notice that this bound is crude in as much as we can get better order lower terms.
Since $d(u, v)\leq \min\{d(u), d(v)\}$, then 
\begin{align*}
P_5(G)\leq \frac{1}{2}\sum_{u\in V(G)}\sum_{u\neq v\in V(G)}d(u)d(v)\min\{d(u), d(v)\}. \end{align*}
So if $(x_1, x_2, x_3, \dots, x_n)$ is the degree sequence of $G$, arranged in decreasing order, we have that 
\begin{align*}
P_5(G) \leq \sum_{i=1}^{n-1}\sum_{j=i+1}^{n}x_ix_j^2.
\end{align*}
To prove Theorem \ref{4}, we need the following lemma. 
\begin{lemma}\label{lm1}
Let $n\geq k\geq 3$ and let $G$ be a planar graph on $n$ vertices such that $S\subseteq V(G)$ with $|S|=k$. Then
$$\sum_{v\in S}d(v)\leq 2n+6k-16.$$
\end{lemma}
\begin{proof}
Let $G'$ be the graph induced by $S$.  Since $G'$ is planar and is not $K_2$, $$\sum\limits_{v\in S}d_{G'}(v)\leq 6k-12.$$
Now we count the number of edges between the vertex sets $S$ and $V(G)\setminus S$, say $e=e(S,V(G)\setminus S)$; that is, the number of edges in the planar bipartite graph with color classes $S$ and $V(G)\setminus S$. Since the graph is bipartite, it is also triangle-free.  Thus, each non-exterior face uses at least $4$ edges. In the case of the exterior face, bridges count twice when counting the number of edges that border the face. So the exterior face has length at least $4$ unless the graph has only one edge.

Hence, if $e>1$, then $4f\leq 2e$, where $f$ is the number of faces in the bipartite subgraph.  Using the inequality and Euler's formula, $n+f=e+2$, we obtain $e=e(S,V(G)\setminus S)\leq 2n-4$.  Therefore, \begin{displaymath}
\sum_{v\in S}d(v)=\sum\limits_{v\in S}d_{G'}(v)+e(S,V(G)\setminus S)\leq 2n+6k-16.\end{displaymath}
If $e=1$, then $\sum_{v\in S}d(v)\leq 6k-11\leq 2n+6k-16$ because $n\geq 3$.
\qedhere
\end{proof}
Given $n\geq 3$, we define the set 
\begin{align*}
A_n = \bigg\{(x_1, x_2, x_3,\dots, x_n) \in \Z^n :~ & n\geq x_1\geq x_2 \geq \cdots \geq x_n \geq 0, \forall k \in \{3, \dots, n\}, \\& \sum_{i=1}^k x_i \leq 2n+6k-16 \ \mbox{ and } \sum_{i=1}^{n}x_i\leq 6n-12\bigg\}.
\end{align*}
Let $(x_1, x_2, \dots, x_n)$ be the degree sequence of an $n$-vertex planar graph $G$ in decreasing order. Since $\sum\limits_{v\in V(G)}d(v) = 2\abs{E(G)}\leq  6n-12$, by Lemma~\ref{lm1}, we have $(x_1, x_2, \dots, x_n)\in A_n$.

Consider the function $S_n:\R^n \to \R$ by $$S_n(x_1, x_2, \dots, x_n) =  \sum_{i}^{n-1}\sum_{j=i+1}^{n}x_ix_j^2,$$ 
then Theorem~\ref{4} will be a corollary of the following theorem.

\begin{theorem}
For $n\geq 3$ and every $(x_1, x_2, \dots, x_n) \in A_n$, we have \begin{displaymath}
S_n(x_1, x_2, \dots, x_n) \leq n^3 +O(n^2).
\end{displaymath}
\label{bound}
\end{theorem}

Before proving Theorem~\ref{bound} we need the following lemmas.

\begin{lemma}\label{lm4}
Let $n\geq 3$ and $(x_1, x_2, \dots, x_n)\in A_n$ be a point maximizing $S_n$ over $A_n$.
Then $x_1-x_2\leq 1$.
\end{lemma}
\begin{proof}
Suppose by contradiction that $x_1-x_2\geq 2$.  Define the sequence $(y_1, y_2, \dots, y_n)\in A_n$ as $y_1=x_1-1$, $y_2=x_2+1$ and $y_i=x_i$ for all $i\neq 1,2$.  Then 
\begin{align*}
S_n(y_1, y_2, \dots, y_n)&=\sum_{i=1}^{n-1}\sum_{j=i+1}^{n}y_iy_j^2=(x_1-1)((x_2+1)^2+x_3^2+\cdots+x_n^2)+(x_2+1)(x_3^2+\cdots+x_n^2)\\&\hspace{12pt}+x_3(x_4^2+\cdots+x_n^2)+\cdots+x_{n-1}x_n^2\\&=x_1(x_2^2+\cdots+x_n^2)+x_2(x_3^2+\cdots+x_n^2)+\cdots+x_{n-1}x_n^2+(x_1-1)(2x_2+1).
\end{align*}
Thus $S_n(y_1, y_2, \dots, y_n)-S_n(x_1, x_2, \dots, x_n)=(x_1-1)(2x_2+1)>0$, which is a contradiction. 
\end{proof}

\begin{lemma}\label{n}
Let $n\geq 3$ and $(x_1, x_2, \dots, x_n)\in A_n$ be a point maximizing $S_n$ over $A_n$.
If $x_1=n$, then $S_n(x_1, x_2, \dots, x_n)\leq n^3+O(n^2)$. 
\end{lemma}
\begin{proof}
By Lemma \ref{lm4}, we have $x_2 \in \{n, n-1\}$.
Since $x_1+x_2+x_3\leq 2n+2$, we see $x_3\leq 3$. 
Therefore,
\begin{align*}
S_n(x_1, x_2, \dots, x_n)&=\sum_{i=1}^{n-1}\sum_{j=i+1}^{n}x_ix_j^2\leq n(n^2+
\underbrace{3^2+3^2+\cdots+3^2}_{n-2\  \text{terms}})\\&\hspace{12pt}+n (\underbrace{3^2+3^2+\cdots+3^2}_{n-2\  \text{terms}})+3(\underbrace{3^2+3^2+\cdots+3^2}_{n-3\  \text{terms}})+\cdots+3(3^2)\\&=n^3+9n(n-2)+9n(n-2)+27(n-3)+27(n-4)+\cdots+27\\&=n^3+18n(n-2)+\frac{27}{2}(n-3)(n-2)=n^3+O(n^2). \qedhere
\end{align*}
\end{proof}

\begin{lemma}\label{lm5}
Let $n\geq 3$ and $(x_1 ,x_2 ,\dots, x_n) \in A_n$. If $x_2\leq \frac{n}{18}$, then $S_n(x_1, x_2, \dots, x_n) \leq n^3+O(n^2)$.
\end{lemma}
\begin{proof}
\begin{align*}
S_n(x_1, x_2, \dots, x_n) &\leq\sum\limits_{i}^{n-1}\sum\limits_{j=i+1}^{n}x_ix_j^2\leq \sum\limits_{i=1}^{n-1}\sum\limits_{j=i+1}^{n}x_ix_jx_2=x_2\sum\limits_{i=1}^{n-1}\sum_{j=i+1}^{n}x_ix_j\leq\frac{x_2}{2}\sum\limits_{i,j}x_ix_j\\&=\frac{x_2}{2}\sum\limits_{i}x_i\sum_{j}x_j=\frac{x_2}{2}\left(\sum\limits_{i}x_i\right)^2=\frac{x_2}{2}(6n-12)^2.
\end{align*}
Thus, if $x_2\leq \frac{n}{18}$, then $S_n(x_1, x_2, \dots, x_n)\leq n^3+O(n^2)$.
\end{proof}
We  prove  the following claim, from which Lemma \ref{5} follows.
\begin{claim}\label{cm1}
Let $n\geq 3$ and suppose $(x_1, x_2, \dots, x_n) \in A_n$. 
If $k$ is the smallest integer at least $3$ such that $\sum\limits_{i=1}^{k}x_i=2n+6k-16$, then $x_k\geq 7$ and $x_{k+1}\leq 6$.
\end{claim}
\begin{proof}
Since $\sum\limits_{i=1}^{k} x_i = 2 n + 6k - 16$, we have $\sum\limits_{i=1}^{k-1} x_i + x_k = 2 n + 6 k - 16$. From the definition of $A_n$, $\sum\limits_{i=1}^{k-1} x_i < 2 n + 6 (k - 1) - 16$. Therefore, $2n + 6k - 16 < 2n + 6(k - 1) - 16 + x_k$. Thus, $x_k \geq 7$. Now suppose $x_{k+1} \geq 7$. In that case, $(2n+6k-16)+7 \leq 2n+6(k+1)-16$ which simplifies to $7 \leq 6$, a contradiction. Therefore, $x_{k+1} \leq 6$.
\end{proof}
\begin{lemma}\label{5}
\label{cases}
Let $n\geq 3$ and $(x_1, x_2, \dots, x_n)\in A_n$ be a point maximizing $S_n$ over $A_n$.
One of the following must hold:
\begin{itemize}
\item[$(i)$] $x_1 = n$,
\item[$(ii)$] $x_2 \leq \frac{n}{18}$,
\item[$(iii)$] there exists a $k \leq 11664$ such that $x_i \leq 6$ for $i>k$.
\end{itemize}
\end{lemma}
\begin{proof}
Suppose that $(i)$ and $(ii)$ are false, that is $x_1<n$ and $x_2 > \frac{n}{18}$. Then we have to show that $(iii)$ holds.
If there exists an $r\geq 3$ such that $\sum\limits_{i=1}^r x_i = 2n+6r-16,$ then take $k$ to be the smallest such $r$.  Otherwise, let $k$ be the last index such that $x_k$ is not 0.  If $k<n$ in both cases, we have $x_k > x_{k+1}$, either because of Claim~\ref{cm1} or because $x_k >0 = x_{k+1}$.  Additionally, from Claim~\ref{cm1}, we have that $x_i \leq 6$ for $i>k$.  We are going to prove that $k\leq 11664$, hence $k$ satisfies $(iii)$.  

Define $y = (y_1, y_2, y_3, \dots, y_n)$ by $y_1=x_1+1$, $y_k=x_k-1$ and, $y_i=x_i$ for $i\neq 1,k$.  
And note $y \in A_n$. 
We have
\begin{align*}
    S_n(y_1, y_2, \dots, y_n)&=\sum_{i=1}^{n-1}\sum_{j=i+1}^{n}y_iy_j^2 \\
    &= (x_1+1)\left(x_2^2+x_3^2+\cdots+x_{k-1}^2+(x_k-1)^2+x_{k+1}^2+\cdots+x_n^2\right) \\
    &\hspace{12pt} +x_2\left(x_3^2+x_4^2+\cdots+x_{k-1}^2+(x_k-1)^2+x_{k+1}^2+\cdots+x_n^2\right) \\
    &\hspace{12pt} +\cdots+x_{k-1}\left((x_k-1)^2+x_{k+1}^2+\cdots+x_n^2\right) \\
    &\hspace{12pt} +(x_k-1)\left(x_{k+1}^2+\cdots+x_n^2\right)+\cdots+x_{n-1}x_n^2 \\
    &= \sum\limits_{i=1}^{n-1}\sum_{j=i+1}^{n}x_ix_j^2+(1-2x_k)(1+x_1+x_2+x_3+\cdots+x_{k-1}) \\
    &\hspace{12pt} +\left(x_2^2+x_3^2+\cdots+x_n^2\right) - \left(x_{k+1}^2+x_{k+2}^2+\cdots+x_n^2\right).
\end{align*}
Thus, $S_n(y_1, y_2, \dots, y_n)-S_n(x_1, x_2, \dots, x_n)=(x_2^2+x_3^2+\cdots+x_k^2)-(2x_k-1)(1 + x_1 + x_2 + \cdots + x_{k-1})$. Since $S_n(y_1, y_2, \dots, y_n)\leq S_n(x_1,x_2,\dots,x_n)$, $x_2>\frac{n}{18}$ and $\sum\limits_{i=1}^{k}x_i\leq 6n$, we have 
\begin{displaymath}
\frac{n^2}{18^2}< (x_2^2 + x_3^2 + x_4^2 + \cdots+x_k^2)\leq (2x_k-1)(1 + x_1 + x_2 + x_3 + \cdots + x_{k-1})<12nx_k.
\end{displaymath}
Therefore, $x_k>\frac{n}{18^2\cdot12}$. Hence we have $6n\geq \sum\limits_{i=1}^kx_i\geq k\frac{n}{18^2\cdot6}$ and $k \leq  (18\cdot 6)^2 = 11664$.
\end{proof}

\begin{lemma}
\label{lm6}
Let $m\geq 2$ be an integer and $x_1, x_2, \dots, x_m$ be reals such that $x_1 \geq x_2 \geq \cdots \geq x_m \geq 0$. Put $t := \sum\limits_{i=1}^m x_i$, then $S_m(x_1,x_2,x_3,\ldots,x_m) \leq (t/2)^3$.
\end{lemma}
\begin{proof}
We are going to proceed by induction on $m$.  First, we show the relation holds for $m=2$.
Let $x_1, x_2$ be real numbers such that $x_1\geq x_2 \geq 0$ and $t=x_1+x_2$, which gives $x_2=t-x_1$.  Hence, $S_2(x_1,x_2) = S(x_1,t-x_1) = x_1(t-x_1)^2$.

Let $f(x) = x(t-x)^2$.  We have $f'(x) = t^2 - 4 t x + 3 x^2 = (t-x)(t-3x)$, which is negative in $[t/2, t]$.  Since $x_1 \geq t/2$, we have  
\begin{align*}
    S_2(x_1,x_2) = f(x_1) \leq \max_{t/2 \leq x \leq t}f(x) = f\left(t/2\right) = \frac{t^3}{8}.
\end{align*}
Therefore, the lemma holds for $m=2$.

Now suppose $m\geq 3$ is such that the lemma is true for $m-1$, and let $x_1, x_2, \dots, x_m$ be real numbers such that $x_1\geq x_2\geq\cdots \geq x_m\geq 0$ and $\sum_{i=1}^mx_i = t$. 
By the induction hypothesis, we have $S_{m-1}(x_1, x_2, \dots, x_{m-1}) \leq \left(\frac{t-x_m}{2} \right)^3.$  
Thus, we get
\begin{align*}
    S_m(x_1,x_2,\dots,x_m) & =S_{m-1}(x_1,x_2,\dots,x_{m-1})+(x_1+x_2+x_3+\cdots+x_{m-1})x_m^2 \\&=S_{m-1}(x_1,x_2,\dots,x_{m-1})+(t-x_m)x_m^2 \leq \frac{(t-x_m)^3}{8}+(t-x_m)x_m^2. 
\end{align*}

Let $g(x) = \frac{(t-x)^3}{8}+(t-x)x^2$,  then $g''(x) = \frac{11 t - 27 x}{4}$.  
We have that $x_m\leq \frac{t}{m} \leq  \frac{t}{3}$, and $g''(x) \geq \frac{2t}{4} \geq 0$, for $x\leq t/3$.  Thus $g$ is convex in $[0,t/3]$, therefore 
\begin{align*}
    S_m(x_1,x_2,\dots,x_m) \leq g(x_m) \leq \max_{0\leq x \leq t/3} g(x) = \max\left\{g(0),g\left(t/3\right)\right\} = \max\left\{\frac{t^3}{8},\frac{t^3}{9}\right\} = \frac{t^3}{8}.\tag*{\qedhere} 
\end{align*}
\end{proof}

Now we are able to prove Theorem~\ref{bound}.

\begin{proof}[Proof of Theorem~\ref{bound}]
Let $n$ be sufficiently large and take $(x_1, x_2, \dots, x_n)$ maximizing $S_n$ over $A_n$.  
By Lemma~\ref{cases} we have three possible cases.

If $x_1 = n$, then $S_n(x_1, x_2, \dots, x_n) \leq n^3 + O(n^2)$ by Lemma~\ref{n}.

If $x_2 \leq \frac{n}{18}$, then $S_n(x_1, x_2, \dots, x_n) \leq n^3 + O(n^2)$ by Lemma~\ref{lm5}.

If there exists a $k$ satisfying $(iii)$ in Lemma~\ref{lm5}, then we have that $\sum\limits_{i=1}^k x_i \leq 2n+6k = 2n + O(1)$.  Hence by Lemma~\ref{lm6}, we have $S_k(x_1,x_2,\dots,x_k) \leq \left(\dfrac{2n+O(1)}{2}\right)^3 = n^3 + O(n^2).$
Therefore, together with the fact that $x_i \leq 6$ for $i>k$, 
\begin{align*}
S_n(x_1,x_2,\dots,x_n) &= \sum_{i=1}^{n-1}\sum_{j=i+1}^n x_ix_j^2 \leq   \sum_{i=1}^{n-1}\left(\sum_{j=i+1}^k x_ix_j^2 + \sum_{j=k+1}^n x_ix_j^2\right) \\ &=  \sum_{i=1}^{k-1}\sum_{j=i+1}^k x_ix_j^2 + \sum_{i=k}^{n-1}\sum_{j=i+1}^k x_ix_j^2 + \sum_{i=1}^{n-1}\sum_{j=k+1}^n x_ix_j^2 \\ &\leq S_k(x_1,x_2,\dots,x_k) + O(n^2) + 36n\sum_{i=1}^{n-1}x_i \\ &\leq  n^3 + O(n^2). \qedhere
\end{align*}
\end{proof}
\section{ Conjectures and Concluding Remarks.}
Let $P_{k+1}$ be a path of length $k$. We propose the following conjectures of the asymptotic values of $f(n,P_{2\ell+1})$ and $f(n,P_{2\ell+2})$.  
\begin{conjecture}\label{gff}
For paths with even length,  $f(n,P_{2\ell+1})=4\ell\left(\frac{n}{\ell}\right)^{\ell+1}+O(n^{\ell}).$
\end{conjecture}
\begin{conjecture}\label{ghh}
For paths with odd length, 
$f(n,P_{2\ell+2})=8\ell(\ell+1)\left(\frac{n}{\ell}\right)^{\ell+1}+O(n^{\ell})$, for $\ell\geq 2.$
\end{conjecture}
For $\ell\geq 3$, in both cases the lower bound is attained by a planar graph on $n$ vertices that is obtained from a balanced blowing up of a maximum independent set of vertices of a $2\ell$-vertex cycle and joining the vertices of each blown-up set by path, see Figure  \ref{fig2}. In the case of $\ell =2$, the lower bound is attained by an $n$-vertex planar graph given in Figure \ref{fig1}.\\
 
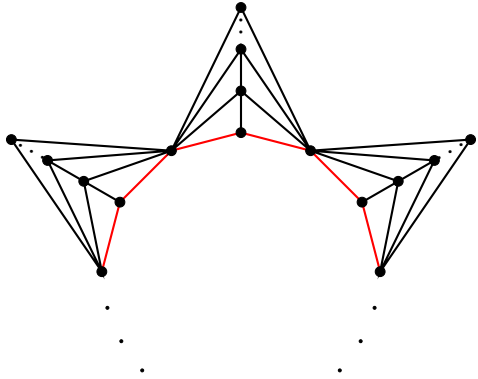
\begin{figure}[h]
\centering
\begin{tikzpicture}[scale=0.185]
\draw[thick,red](10,0)--(8.7,5)--(5,8.7)--(0,10)--(-5,8.7)--(-8.7,5)--(-10,0);
\draw[thick](5,8.7)--(11.3,6.5)--(10,0)--(13.9,8)--(5,8.7)--(16.5,9.5)--(10,0)(8.7,5)--(13.9,8);
\draw[thick](-5,8.7)--(-11.3,6.5)--(-10,0)--(-13.9,8)--(-5,8.7)--(-16.5,9.5)--(-10,0)(-8.7,5)--(-13.9,8);
\draw[thick](5,8.7)--(0,13)--(-5,8.7)--(0,16)--(5,8.7)--(0,19)--(-5,8.7)(0,10)--(0,16);
\draw[fill=black](0,19)circle(10pt);
\draw (-15.2,8.75) node[rotate=150]{$\dots$};
\draw (15.2,8.75) node[rotate=30]{$\dots$};
\draw (0,17.4) node[rotate=90]{$\dots$};
\draw[fill=black](16.5,9.5)circle(10pt);
\draw[fill=black](-16.5,9.5)circle(10pt);
\draw[fill=black](10,0)circle(10pt);
\draw[fill=black](-10,0)circle(10pt);
\draw[fill=black](0,10)circle(10pt);
\draw[fill=black](0,13)circle(10pt);
\draw[fill=black](0,16)circle(10pt);
\draw[fill=black](8.7,5)circle(10pt);
\draw[fill=black](11.3,6.5)circle(10pt);
\draw[fill=black](13.9,8)circle(10pt);
\draw[fill=black](-8.7,5)circle(10pt);
\draw[fill=black](-11.3,6.5)circle(10pt);
\draw[fill=black](-13.9,8)circle(10pt);
\draw[fill=black](5,8.7)circle(10pt);
\draw[fill=black](-5,8.7)circle(10pt);
\draw[fill=black](-9.6,-2.6)circle(3pt);
\draw[fill=black](9.6,-2.6)circle(3pt);
\draw[fill=black](-8.6,-5)circle(3pt);
\draw[fill=black](8.6,-5)circle(3pt);
\draw[fill=black](-7.1,-7.1)circle(3pt);
\draw[fill=black](7.1,-7.1)circle(3pt);
\end{tikzpicture}
\caption{The graph obtained by blowing up every other vertex in an even cycle and joining the copies of the vertices by a path. This graph attains the lower bound stated in Conjecture \ref{gff} and \ref{ghh}.}
\label{fig2}
\end{figure}

\section*{Acknowledgements}

The research of the second, the fifth and the seventh authors is partially supported by the National Research, Development and Innovation Office -- NKFIH, grant K 132696. 
The research of the third author was partially supported by Simons Foundation Collaboration Grant \#353292 and by the J. William Fulbright Educational Exchange Program.
The research of the fifth author is partially supported by  Shota Rustaveli National Science Foundation of Georgia SRNSFG, grant number DI-18-118.
\newpage

\bibliographystyle{abbrv}
\bibliography{main}

\end{document}